 \documentclass[preprint,10pt]{elsarticle}

 \usepackage{graphicx}

\usepackage{amssymb}
\usepackage{amsfonts,amsmath,amssymb,amscd}
\usepackage{amsthm}
\usepackage{txfonts}
 \usepackage[english]{babel}
\usepackage{amstext}
\usepackage{mathrsfs}
 \theoremstyle{plain}
 \newtheorem{theorem}{Theorem}[section]

 \newtheorem{lemma}{Lemma}[section]

\theoremstyle{definition}
 \newtheorem{definition}{Definition}[section]
 \newtheorem{example}{Example}[section]
 \newtheorem{remark}{Remark}
\theoremstyle{remark}
\usepackage[english]{babel}
\usepackage{amssymb,amstext,amsmath}
\usepackage{amsfonts,amsmath,amssymb,amscd}
\usepackage{amsfonts}
\usepackage{mathrsfs}
 \usepackage{lineno}

\journal{ELSEVIER}

\begin{document}

\begin{frontmatter}


 \title{Remarks on martingale representation theorem for set-valued martingales}

 \author[label1]{Jinping Zhang}
 \address[label1]{School of Mathematics and Physics, North China Electric Power University,
Beijing, 102206, P.R.China}
\address[label2]{Graduate School of Science, Kyoto University, Kyoto, 606-8501, Japan }
\ead{zhangjinping@ncepu.edu.cn}
 \author[label2]{Kouji Yano}
 \ead{kyano@math.kyoto-u.ac.jp}
\begin{abstract}
Martingale representation theorem for set-valued martingales was proposed by M. Kisielewicz [J. Math. Anal. Appl. 2014]. We shall prove that the result holds only for very special case: the set-valued martingale degenerates to the point-valued one. A revised representation theorem for a special kind of non-degenerate set-valued martingales is presented.
\end{abstract}


\end{frontmatter}


\section{Introduction}
\label{author_sec:1}
 Set-valued function is a natural extension of point-valued one, which has received much attention. A set is called {\it degenerate} if it is a singleton and otherwise {\it non-degenerate}. A difficulty to handle non-degenerate set-valued functions lies in the fact that the power set of a set is not linear. For example, let $(\Omega, \mathscr A, P)$ be a probability space, $\lambda, \eta, a,b, c,d\in\mathbb R$ be real numbers with $a<b, c<d$ and $f,g$ ($f\leq g $ a.s.) be two integrable random variables taking values in $\mathbb R$.   Define the operations $\lambda[a,b]:=\left \{\lambda h: a\leq h\leq b\right \}$ and $[a,b]+[c,d]:=[a+c, b+d]$. We then remark that $(\lambda+\eta)[a,b]\neq\lambda [a,b]+\eta[a,b]$ if $\lambda$ and $\eta$ have different signs. For the expectation (which will be defined by \eqref{eqn:expectation}) $E(\lambda[f,g])=\lambda E([f,g])$ holds, while $E(f[a,b])=E(f)\times [a,b]$ may not hold.  We have to be very careful to deal with set-valued variables.

 Martingales are a very important class of stochastic processes with nice properties and wide applications. For set-valued case, Hiai and Umegaki 1977 \cite{Hia} defined set-valued martingale, supermartingale and submartingale. After that, there have been many works studying set-valued martingales. For example, based on the definition of set-valued stochastic integral with respect to Brownian motion given by Jung and Kim \cite{Jun}, Zhang et al. \cite{Zha09,zhang20} obtained the submartingale property of set-valued stochastic integrals with respect to Brownian motion and to compensated Poisson random measure. Furthermore, in \cite{zhang20}, by using the Hahn decomposition theorem and the properties of compensated Poisson random measure, the authors proved that the non-degenerate set-valued stochastic integral w.r.t. the compensated Poisson measure is not a set-valued martingale but a submartingle. For the non-degenerate case w.r.t. Brownian motion, in a very simple way, we shall show that it is neither a set-valued martingale.
 The martingale representation theorem for point-valued martingale plays an important role in classical stochastic analysis,  see e.g. Theorem 4.3.4 in \cite{oksendal}. For set-valued case, is there a similar representation theorem? It is our task in this short paper to answer the question.

When taking as the underlying space the $r$-dimensional Euclidean space $\mathbb R^{r}$, Kisielewicz in \cite{Kis2014} (2014) proposed a representation theorem for  set-valued martingale as follows:
\begin{theorem}\label{kis}
 For every set-valued $\mathbb A$-martingale $F=(F_{t})_{t\geq 0}$ defined on a filtered probability space $(\Omega, \mathscr A, \mathbb A, P)$ with the augmented natural filtration $(\mathscr A_{t})_{t\geq 0}$ of an $d$-dimensional Brownian motion $B=(B_{t})_{t\geq 0}$ defined on $(\Omega, \mathscr {A}, P)$ and such that $F_{0}=\left \{0\right \}$ there exists a set ${\mathscr G}\in {\mathscr P} ({\mathcal L}^{2}_{\mathbb A})$ such that $F_{t}=\int_{0}^{t}\mathscr{G} dB_{\tau}$ a.s. for every $t\geq 0$.
\end{theorem}

\noindent Here ${\mathcal L}^{2}_{\mathbb A}$ denotes the family of all $\mathbb R^{r\times d}$-valued integrable (w.r.t. the Brownian motion $B$) stochastic processes and ${\mathscr P} ({\mathcal L}^{2}_{\mathbb A})$ denotes the power set of ${\mathcal L}^{2}_{\mathbb A}$. For $\mathscr G\in {\mathscr P} ({\mathcal L}^{2}_{\mathbb A})$, $\int_{0}^{t}\mathscr{G} dB_{\tau}$ is the generalized stochastic integral defined in \cite{Kis2015}, which is a set-valued random variable for each $t$.

We will see in Example \ref{ex:2} that Theorem \ref{kis} imposes so strong assumption $F_{0}=\{0\}$ as to exclude non-degenerate set-valued martingales. We will propose in Theorem \ref{main} a revised representation theorem.

This paper is organized as follows. In Section 2 we recall several notations and preliminary facts about set-valued processes. In Section 3 we develop our main theorem. We shall prove that the above representation theorem does not hold for non-degenerate set-valued martingale.  A revised martingale representation will be given. Except for a special kind of non-degenerate set-valued martingale, the general non-degenerate set-valued martingale has no representation theorem.
\section{ Notations and Preliminaries}
\label{author_sec:2}

Throughout the paper, let $T$ be a positive number. We consider the complete non-atomic probability space $(\Omega, \mathscr {A}, P)$ equipped with a  filtration  $\mathbb A=(\mathscr {A}_{t})_{t\in [0, T]}$ which satisfies the usual condition. 
We denote by $(\frak X,\| \cdot \|)$ a separable Banach space equipped with the Borel sigma-algebra $\mathscr {B}(\frak X)$.  $K(\frak X)$ denotes the family of all nonempty closed subsets of $\frak X$.  $K_{c}(\frak X)$ (resp. $K_{bc}(\frak X)$, $K_{kc}(\frak X)$) is the family of all nonempty closed convex (resp. nonempty bounded closed convex, nonempty compact convex) subsets of $\frak X$. For a set $C\subset\frak X$, we write $\|C \|:=\sup_{x\in C}\|x\|$. $L^{p}(\Omega, {\mathscr A}, P; \frak X)$=$L^{p}(\Omega; \frak X) $ ($1\leq p<\infty $) is the set of all $\frak X$-valued Borel measurable functions $f:
\Omega\rightarrow \frak X$ equipped with the norm
$
 \|f\|_{p}:=\left \{\int_{\Omega}\|f(\omega)\|^{p}dP(\omega)\right \}^{\frac{1}{p}}<\infty.
$

 A mapping $F: \Omega\rightarrow K(\frak X)$ is said to be a {\it set-valued random variable (or a random set)} if for each open set $O\subset \frak X$, we have $\left \{\omega\in \Omega: F(\omega)\cap O \neq \emptyset\right \}\in \mathscr A$. We denote the family of all $K(\frak X)$-valued random variables  by ${\cal M}(\Omega, {\mathscr A}, P; K(\frak X))$, or briefly by
${\cal M}(\Omega; K(\frak X))$. Similarly, we also have the notations ${\cal M}(\Omega; K_{c}(\frak X))$, ${\cal M}(\Omega; K_{bc}(\frak X))$ and ${\cal M}(\Omega; K_{kc}(\frak X))$.
For $F\in {\cal M}(\Omega, K(\frak X))$, the family of all $L^{p}$-
integrable selections is denoted by

$$
S^{p}_{F}({\mathscr A}):=\left \{f\in L^{p}(\Omega, {\mathscr A}, P; \frak X) : f (\omega)\in
F(\omega) \
 a.s.\right \}, 1\leq p< \infty,
$$
or briefly by
$S_{F}^{p}$.
A set-valued random variable $F$ is said to be {\em integrable} if
$S_{F}^{1}$ is nonempty. $F$ is called {\em $L^{p}$-integrably bounded} if there exists $h\in L^{p} (\Omega, {\mathscr A},
P; \mathbb R)$ s.t. $\|F(\omega)\|\leq h(\omega)$, i.e. [for all $x\in F(\omega)$ we have $\|x\|\leq h (\omega)$] almost surely.
 The family of all $K(\frak X)$-valued $L^{p}$-integrably bounded random variables is
denoted by $L^{p}(\Omega, {\mathscr A}, P; K(\frak X))$, or briefly by $L^{p}(\Omega; K
(\frak X))$.  Similarly, we have notations  $L^{p}(\Omega; K_{c} (\frak X))$, $L^{p}(\Omega; K_{bc} (\frak X))$ and $L^{p}(\Omega; K_{kc} (\frak X))$.

Let $\Gamma$ be a set of measurable functions $f: \Omega \rightarrow
\frak X$. We call $\Gamma$ {\em decomposable} with respect to the
$\sigma$-algebra $\mathscr A$ if for any finite $\mathscr A$-measurable partition
$A_{1},.., A_{n}$ of $\Omega$ and for any $f_{1},..., f_{n}\in \Gamma $ it
follows that $1_{A_{1}}f_{1}+...+1_{A_{n}}f_{n}\in \Gamma$,
where $1_ {A}$ is the indicator function of  $A$. From \cite{Hia}, we know that a nonempty  set $\Gamma$ ($ \subset L^{p}(\Omega, {\mathscr A}, P; \frak X)$) determines a $p$-integrable set-valued random variable $F$  such that $\Gamma=S_ {F}^{p} $
if and only if $\Gamma$ is decomposable with respect to $\mathscr A $. Note also that $F(\omega)=G(\omega)$ a.s. iff  $S_{F}^{p}=S_{G}^{p}$. Therefore, in order to study $F$, we have only to study $S_{F}^{p}$.

There is the Castaing representation for a $p$-integrable set-valued
random variable.

\begin{lemma}[{\cite{Hia}}]\label{lem:2.2}
For a $p$-integrable  set-valued random variable $F\in {\cal M}(\Omega, {\mathscr A}, P; K(\frak X))$, there exists a sequence $\left \{f^{i}: i\in \mathbb N\right \}\subset S_{F}^{p}$ such that
$F(\omega)=cl\left \{f^{i}(\omega): i\in \mathbb N\right \}$ for all $\omega \in
\Omega$, where the closure is taken in $\frak X$. In addition, $S_{F}^{p}=\overline{de}\left \{f^{i};i=1,2,...\right \}$, where the $\overline{de}$ denotes the decomposable closure of the sequence $\left \{f^{i}: i=1,2,...\right \}$ in the space $L^{p}(\Omega; \frak X)$.
\end{lemma}

The {\em integral (or
expectation)} of a set-valued random variable $F$ was defined by
Aumann (\cite{Aum}) as
\begin{equation}\label{eqn:expectation}
E(F):=\left \{E(f): f\in S_{F}^{1}\right \},
\end{equation}
where $E(f)=\int_{\Omega}fdP$ is the Bochner integral.
If the probability space is non-atomic, we have $cl\left \{E(F)\right \}$ is convex. In general, the expectation $E(F)$ is not closed. But under some conditions, it is closed. For example,  if the Banach space $\frak X$ has the Radon Nikodym property, and if $F\in L^{1}(\Omega, K_{kc}(\frak X))$ then $E(F)$ is closed in $\frak X$. Moreover, if the space $\frak X$ is reflexive, and if $F\in L^{1}(\Omega, K_{c}(\frak X))$, then the expectation $E(F)$ is closed in $\frak X$ (\cite{Hia}). It is well-known that finite dimensional spaces and $L^{p} (\Omega; \frak X)\  (1<p<\infty)$ are reflexive.

A set-valued stochastic process $F=\left \{F_{t}: 0\leq t\leq T\right \}$ is called {\em uniformly integrable} if the real-valued stochastic process $\|F\|=\left \{ \|F_{t}\|: 0\leq t\leq T\right \}$ is uniformly integrable.

$F=\left \{F_{t}: 0\leq t\leq T\right \}$ is called an {\it $\mathbb A$-adapted set-valued stochastic process} if $F_{t}\in {\cal M}(\Omega, \mathscr A_{t}, P; K(\frak X))$ for each $t$. $f=\left \{f_{t}: 0\leq t\leq T\right \}$ is called a {\it martingale selection} of $F$ if $f$ is an $\frak X$-valued  $\mathbb A$-martingale and  $f_{t}(\omega)\in F_{t}(\omega)$ a.s. for each $t$. The family of all martingale selections of $F$ is denoted by $MS(F)$. The set-valued  {\it conditional expectation} $E[F_{t}|\mathscr A_{s}]$ is defined by $S^{1}_{E[F_{t}|\mathscr A_{s}]}(\mathscr A_{s}):=cl\left \{E[f_{t}|\mathscr A_{s}]: f_{t} \in S^1_{F_{t}}(\mathscr A_{t})\right \}$ for $0\leq s\leq t$, where the closure is taken in $L^{1}(\Omega; \frak X)$.

\begin{definition}\label{martingale}
An integrable convex set-valued $\mathbb A$-adapted stochastic
process $F=\left \{F_{t}: 0\leq t\leq T \right \}$ is called a {\it set-valued
$\mathbb A$-martingale} if for any $0\leq s \leq t$ it holds that
$E(F_{t}|{\mathscr{A}}_{s})=F_{s}$ in the sense of
$S_{E(F_{t}|{\mathscr{A}}_{s})}^{1}({\mathscr{A}}_{s})=S_
{F_{s}}^{1}({\mathscr{A}}_{s})$.
It is called a {\it set-valued submartingale (resp. supermartingale)} if for any
$0\leq s\leq t$, $E(F_{t}|{\mathscr{A}}_ {s})\supset F_{s}$
(resp. $E[F_{t}|{\mathscr{A}}_{s}]\subset F_{s} $) in the sense
of $S_{E(F_{t}|{\mathscr{A}}_{s})}^{1}({\mathscr{A}}_{s}) \supset S_
{F_{s}}^{1}({\mathscr{A}}_{s})$ (resp.
$S_{E(F_{t}|{\mathscr{A}}_{s})}^{1}({\mathscr{A}}_{s}) \subset S_
{F_{s}}^{1}({\mathscr{A}}_{s})$).
\end{definition}
A set-valued martingale has the  property: $cl\left \{E(F_{t})\right \}=cl\left \{E(F_{0})\right \}$ for every $0\leq t\leq T$.

  In the case $\frak X=\mathbb R$, it is known (Theorem 3.1.1 of \cite{thesis}) that $\left \{F_{t}=[f_{t}, g_{t}]: 0\leq t\leq T \right \}$ is an interval-valued $\mathbb A$-martingale iff both the endpoints are $\mathbb R$-valued $\mathbb A$-martingales. For discrete time, the result holds too.

\section{Main result}

Let $(B_{t})_{0\leq t\leq T}$ be a real-valued standard Brownian motion.  Let $\mathbb A=(\mathscr {A}_{t})_{t\in [0, T]}$ denote the augmented natural filtration of $(B_{t})_{0\leq t\leq T}$. We adopt the notations of Section \ref{author_sec:2}.

\begin{theorem}\label{main}
Suppose $(\frak X, \|\cdot\|)$ is a separable reflexive M-type 2 Banach space.  Let $\left \{M_{t}, 0\leq t\leq T\right \}$ be a  $K_{c}(\frak X)$-valued $\mathbb A$-martingale. Then the following statements are equivalent:

\noindent (i) There exists a set-valued stochastic process $(G_{t})_{0\leq t\leq T}$ such that

\begin{equation}\label{eqn:representation}
M_{t}=E(M_{0})+\int_{0}^{t}G_{s}dB_{s}\ \  a.s. \ for  \ each \ t.
\end{equation}

\noindent (ii) There exists an $\frak X$-valued stochastic process $g=\left \{g_{t}: t\in[0, T]\right \}$  and a bounded, closed and convex subset $C\subset \frak X$ such that $$M_{t}=C+\left\{\int_{0}^{t}g_{s}dB_{s}\right\}\  a.s. \ for \ each \ t.$$

\noindent (iii)
There exists a sequence  $\left \{f^{1},\cdots,f^{n},\cdots\right \}$ of $\frak X$-valued $\mathbb A$-martingales such that
 $$M_{t}=cl\left\{f^{1}_{t},\cdots, f^{n}_{t},\cdots\right\} \ a.s. \ for \ each \ t$$
 and  $f_{t}^{i}-f^{j}_{t}$ is non-random and independent of $t$ for any $i, j\geq 1$.

\end{theorem}
\begin{remark} For the M-type 2 Banach space, the set-valued stochastic integral w.r.t. Brownian motion $I_{t}(G)=\int_{0}^{t}G_{s}dB_{s}$ is defined in Definition 4.2  in \cite{Zha09}, which is a $K(\frak X)$-valued random variable for each $ t$. $I_{t}(G)$ is determined by $$S_{I_{t}(G)}^{2}(\mathscr A_{t}):=\overline{de}\left\{\int_{0}^{t}g_{s}dB_{s}: g \ is \ the \ integrable \ selection\ of \ G \right\}.$$  If $\frak X$ is not reflexive,  replacing \eqref{eqn:representation} by
 $M(t)=cl\left \{E(M_{0})\right \}+\int_{0}^{t}G_{s}dB_{s} \  $, the result still holds. It is also valid for  $\frak X=\mathbb R^{r}$ with $d$-dimensional Brownian motion and corresponding $K(\mathbb R^{r\times d})$-valued integrable stochastic process $G$.
\end{remark}
In the following, the family of all $\frak X$-valued square-integrable (w.r.t  Brownian motion) stochastic processes is denoted by ${\cal L}^{2}(\frak X)$. The family of all $K(\frak X)$-valued integrable (w.r.t. Brownian motion) is denoted by ${\cal L}^{2}(K(\frak X))$.
Before prove Theorem \ref{main}, firstly, we give a result about expectation of set-valued random variable.
\begin{lemma}\label{mean0}
For any set-valued random variable $F\in L^{1}(\Omega, K(\frak X))$ and any deterministic element $a\in\frak X$, the expectation $E(F)=\left \{a\right \}$  iff $F$ degenerates to a random singleton $\left \{f\right \}$ with $E(f)=a$.
\end{lemma}
\begin{proof}
We may assume $a=0$ without loss of generality thanks to translation. The
 sufficiency is obvious. For the converse,  assume $F\in L^{1}(\Omega, K(\frak X))$ is non-degenerate. By the Castaing representation Lemma \ref{lem:2.2}, there exist at least two different $\frak X$-valued functions $f_{1}, f_{2}\in S_{F}^{1}$. If one of $f_{1}, f_{2}$ is not mean zero, then the result is obtained. We assume $E(f_{1})=E(f_{2})=0$ and $f_{1}\neq f_{2}$ with a positive probability. Since $(\Omega, {\mathscr A}, P)$ is non-atomic, there exists a measurable partition $\left \{A, B\right \}$ of $\Omega$ such that $0<P(A)<1$ and
\begin{equation*}
\begin{split}
E(F)&\ni E(1_{A}f_{1}+1_{B}f_{2})=E(1_{A}f_{1})+E(1_{B}f_{2})\\
&=E(1_{A}f_{1})-E(1_{A}f_{2})  \ \ \ (since \ E(f_{2})=E(1_{A}f_{2}+1_{B}f_{2})=0)\\
&=E(1_{A}(f_{1}-f_{2}))\neq 0.
\end{split}
\end{equation*}
\end{proof}


 By using Lemma \ref{mean0}, it is convenient to judge some set-valued stochastic processes are not set-valued martingales. There are some examples.

 \begin{example}\label{ex:1}
 Let $G\in{\cal L}^{2}(K_{c}(\frak X))$ be non-degenerate. Then the set-valued stochastic integral $\left \{\int_{0}^{t}G_{s}dB_{s}: t\in[0,T]\right \}$ is not a set-valued martingale.

  In fact, by the definition of $I_{t}(G)=\int_{0}^{t}G_{s}dB_{s}$, we know $I_{0}(G)=\left \{0\right \}$ a.s. Then $E(I_{0}(G))=\left \{0\right \}$. But there exists $t>0$ such that $I_{t}(G)$ is non-degenerate. By Lemma \ref{mean0}, $E(I_{t}(G))\neq E(I_{0}(G))$, then $\left \{\int_{0}^{t}G_{s}dB_{s}: t\in[0,T]\right \}$ is not a set-valued martingale.

  In \cite{Zha09}, we know that it is a set-valued submartingale.
\end{example}

\begin{example}\label{ex:2}
 Now let us concentrate on the martingale representation Theorem \ref{kis}.

   By using Lemma \ref{mean0}, it is easy to see that there is no non-degenerate set-valued martingale such that its expectation is a singleton. Particularly, there is no non-degenerate  set-valued martingale with expectation zero. From this point of view, the martingale representation Theorem \ref{kis} given by  Kisielewicz in \cite{Kis2014} does not hold for any non-degenerate set-valued martingale since $F_{0}=\left \{0\right \}$ a.s. In fact, in  Theorem \ref{kis}, if the subset $\mathscr G$ is not a singleton, then the  integral process $\left \{I_{t}(\mathscr G):=\int_{0}^{t}{\mathscr G}dB_{\tau}: 0\leq t\leq T\right \}$ is non-degenerate. $\left \{I_{t}(\mathscr G): 0\leq t\leq T\right \}$ is not a non-degenerate set-valued martingale since $I_{0}(\mathscr G)=\left \{0\right \}$ a.s.

  If the set $\mathscr G$ only includes one $\frak X$-valued integrable stochastic process, Theorem \ref{kis} becomes the classical martingale representation theorem.
  \end{example}

 \begin{example}\label{eg1}
Take two different stochastic processes $f,g\in {\cal L}^{2}(\mathbb R)$ and set $$\xi_{t}=\int_{0}^{t}f_{s}dB_{s}, \ \ \eta_{t}=\int_{0}^{t}g_{s}dB_{s}.$$ For each $t\in[0,T]$, define
$$M_{t}=cl\left\{\lambda \xi_{t}+(1-\lambda)\eta_{t}: \lambda\in[0,1]\cap\mathbb Q\right\},
$$
i.e., $M_{t}$ is the segment of $\int_{0}^{t}f_{s}dB_{s}$ and $\int_{0}^{t}g_{s}dB_{s}$. It is clear that $\left \{M_{t}: t\in[0,T]\right \}$  is a non-degenerate interval-valued stochastic process. But it is not an interval-valued martingale since $M_{0}=\left \{0\right \}$ a.s.

We know that,  $M_{t}=\left[\min\left \{\xi_{t}, \eta_{t}\right \}, \max\left \{\xi_{t}, \eta_{t}\right \}\right]$, and that it is a non-degenerate interval-valued submartingale.
In fact, the process $\min\left \{\xi_{t}, \eta_{t}\right \}$ is a real-valued supermartingale and the process $\max\left \{\xi_{t}, \eta_{t}\right \}$ is a real-valued submartingale.
So that $E\left(M_{t}|\mathscr A_{s}\right)\supset M_{s}$ for $0\leq s\leq t\leq T$.
\end{example}

 \begin{remark}
 For discrete time set-valued martingales, Ezzaki and Tahri (Corollary 3.8 of \cite{Ezzaki}) proposed that in a separable RNP (Radon Nikodym Property) Banach space $\frak X$, a sequence $\left \{F_{n}, n=1,2,\cdots\right \}$ is a uniformly integrable $K_{bc}(\frak X)$-valued martingale if and only if it admits a Castaing representation of regular martingale selections such that $ {\lim\inf}_{n\rightarrow\infty}\int_{\Omega} \|F_{n}\|dP<\infty$. Unfortunately, the `if' part may not hold. For example, let $\left \{f_{n}, n=1,\cdots\right \}$ be a  uniformly integrable $\mathbb R$-valued martingale. For each $n$, we assume that the value of $f_{n}$ changes its signs with positive probability. Set $M_{n}=cl\left \{\lambda f_{n}+2(1-\lambda)f_{n}; \lambda\in [0, 1]\cap \mathbb Q\right \}$. Then $M_{n}=[\min\{f_{n}, 2f_{n}\}, \max\{f_{n}, 2f_{n}\}]$. As a manner similar to Example \ref{eg1}, we know that $\left \{M_{n}; n=1,\cdots\right \}$ is not an interval-valued martingale. But $ cl\left \{\lambda f_{n}+2(1-\lambda)f_{n}; \lambda\in [0, 1]\cap \mathbb Q\right \}$ is a Castaing representation of regular martingale selections with
 $$ {\lim\inf}_{n\rightarrow\infty}\int_{\Omega}\|M_{n}\|dP\leq 2{\lim\inf}_{n\rightarrow\infty}\int_{\Omega}|f_{n}|dP<\infty.$$
 \end{remark}

{\bf Proof of Theorem \ref{main}:}
\begin{proof}
(i)$\Longrightarrow$ (ii): Assume there exists $G\in {\mathcal L}^{2}(K(\frak X))$ such that $M_{t}=E(M_{0})+\int_{0}^{t}G_{s}dB_{s}$ a.s. for each $t$. Then
$E(M_{t})=cl\left \{E(M_{0})+E(\int_{0}^{t}G_{s}dB_{s})\right \}=E(M_{0})$, which implies $E\left(\int_{0}^{t}G_{s}dB_{s}\right)=\left \{0\right \}$.
By Lemma \ref{mean0}, we obtain that $G$ is a degenerate set-valued stochastic process, which is denoted by $\left \{g\right \}$ ($ g\in {\cal L}^{2}(\frak X)$ ). Hence (ii) holds with $C=E(M_{0})$.

(ii)$\Longrightarrow$ (iii): Since $E(M_{0})$ is a  closed  subset of the separable Banach space $\frak X$, there exists a sequence $\left \{x_{1},\cdots,x_{n},\cdots\right \}\subset \frak X$ such that $E(M_{0})=cl \left \{x_{1},\cdots, x_{n},\cdots\right \}$. Then we have
\begin{equation*}
\begin{split}
M_{t}&=cl \left \{x_{1},\cdots, x_{n},\cdots\right \}+\left \{\int_{0}^{t}g_{s}dB_{s}\right \}\\
&=cl \left \{x_{1}+\int_{0}^{t}g_{s}dB_{s},\cdots, x_{n}+\int_{0}^{t}g_{s}dB_{s},\cdots\right \}
=:cl\left \{f^{1}_{t},\cdots,f^{n}_{t},\cdots \right \}.
\end{split}
\end{equation*}
Apparently, such a sequence $\left \{f^{1}_{t},\cdots,f^{n}_{t},\cdots \right \}$ satisfies the requirement.

(iii)$\Longrightarrow$ (i):
Assume there exists an $\frak X$-valued $\mathbb A$-martingale sequence $\left \{f^1,\cdots,f^n,\cdots\right \}$ such that $M_{t}=cl\left \{f_{t}^1,\cdots,f_{t}^n,\cdots\right \}$ a.s. for all $t$. Then by the classical martingale representation theorem, there exists $g\in {\cal L}^{2}(\frak X)$ such that $f^{1}_{t}=E(f^{1}_{0})+\int_{0}^{t}g_{s}dB_{s}$. Notice that for any $i, j$, the difference $f^{i}-f^{j}$ is non-random and independent of $t$, then $f^{n}_{t}=E(f^{n}_{0})+\int_{0}^{t}g_{s}dB_{s}$ for each $n$. Thus
\begin{equation*}
\begin{split}
M_{t}&=cl\left \{E(f^{1}_{0})+\int_{0}^{t}g_{s}dB_{s},\cdots, E(f^{n}_{0})+\int_{0}^{t}g_{s}dB_{s},\cdots\right \}\\
&=cl\left \{E(f^{1}_{0}),\cdots,E(f^{n}_{t}),\cdots\right \}+\left \{\int_{0}^{t}g_{s}dB_{s}\right \}=E(M_{0})+\int_{0}^{t}G_{s}dB_{s},
\end{split}
\end{equation*}
where $G=\left \{g\right \}$. In fact, each $f^{i}_{0}$  is constant a.s  since $\mathscr A_{0}$ is trivial.
 \end{proof}

Interval-valued martingale  is concrete and easy to deal with. It is sufficient to consider the endpoints. Now we give the representation theorem of interval-valued martingale and some examples when taking $\frak X=\mathbb R$.

\begin{theorem}
Let $M=\left \{M_{t}=[a_{t}, b_{t}]:0\leq t\leq T\right \}$ be a $K_{c}(\mathbb R)$-valued $\mathbb A$-martingale. Then there exists an interval-valued stochastic process $G\in{\cal L}^2(K(\mathbb R))$ such that
\begin{equation}\label{eqn:interval1}
M_{t}=[E(a_{0}), E(b_{0})]+\int_{0}^{t}G_{s}dB_{s}, \ \ for \ every \ 0\leq t\leq T
\end{equation}
if and only if $b_{t}-a_{t}$ is a constant a.s. for each $t$ and $G$ degenerates to a singleton process $\left \{g\right \}$ with $g\in {\cal L}^{2}(\mathbb R)$.
\end{theorem}
\begin{proof}
This is a direct consequence of Theorem \ref{main}.

\end{proof}

\begin{example}
Let $B=\left \{B_{t}: 0\leq t\leq T\right \}$ be the real-valued Brownian motion as above. For $t\in[0,T]$, set $M_{t}=\left[e^{B_{t}-\frac{t}{2}}, 1+e^{B_{t}-\frac{t}{2}}\right]$, then $\left \{M_{t}: t\in[0,T]\right \}$ is an interval-valued martingale with the following martingale representation
$$
M_{t}=[1,2]+\left \{\int_{0}^{t}e^{B_{s}-\frac{s}{2}}dB_{s}\right \} \  a.s. \ for\ all\ t.
$$

 In fact, $E\left(\int_{0}^{T}(e^{B_{t}-\frac{t}{2}})^{2}dt\right)=e^{T}-1<\infty$, then $\left \{e^{B_{t}-\frac{t}{2}}: t\in[0,T]\right \}\in {\cal L}^{2}(\mathbb R)$.  By Ito's formula,
$$d(e^{B_{t}-\frac{t}{2}})=-\frac{1}{2}e^{B_{t}-\frac{t}{2}}dt+e^{B_{t}-\frac{t}{2}}dB_{t}+\frac{1}{2}e^{B_{t}-\frac{t}{2}}dt=e^{B_{t}-\frac{t}{2}}dB_{t}.
$$
Then $e^{B_{t}-\frac{t}{2}}=1+\int_{0}^{t}e^{B_{s}-\frac{s}{2}}dB_{s}>0$ and $E\left(e^{B_{t}-\frac{t}{2}}\right)=1$. Thus $E(M_{t})=[1,2]$.
\end{example}
\begin{example}
Setting $M_{t}=\left[e^{B_{t}-\frac{t}{2}}, 2e^{B_{t}-\frac{t}{2}}\right]$, then $\left \{M_{t}: t\in[0,T]\right \}$ is an interval-valued martingale since both the endpoints are real-valued martingales and $E(M_{t})=[1,2]$. But $M_{t}\neq [1, 2]+\int_{0}^{t}\left[e^{B_{s}-\frac{s}{2}}, 2e^{B_{s}-\frac{s}{2}}\right]dB_{s}$.  $M_{t}$ does not have the martingale representation since the difference between endpoints of $M_{t}$ is $e^{B_{t}-\frac{t}{2}}$, not a constant.
\end{example}

 Interval-valued martingale is a concrete example of set-valued martingales with wide applications in statistical modelling and practical fields such as econometrics, mathematical finance etc. For example, Sun et al. \cite{sun} proposed  a threshold autoregressive model based on interval-valued data. In the model,  $\left \{u_{t}: t\in[0,T]\right \}$  is an interval-valued martingale difference sequence(the difference between two intervals is the Hukuhara difference which guarantees the difference of two identical sets is zero).  There is an assumption in the model: $E(u_{t}|I_{t-1})=[0,0]$ where $I_{t-1}$ is the information set up to time $t-1$. Unfortunately, the assumption is not appropriate  for non-degenerate interval-valued stochastic process.

\section*{Acknowledgment}
 We finished the draft of this paper when Jinping Zhang stayed in Kyoto university as a visiting scholar. We would like to thank Kosuke Yamoto and Toru Sera for their valuable academic discussion. The research of Jinping Zhang was supported by National Science Foundation of Beijing Municipality (No.1192015) and the Construct Program of the Key Discipline in Hunan Province. The research of Kouji Yano was supported by JSPS KAKENHI grant No.'s JP19H01791, JP19K21834 and JP18K03441 and by JSPS Open Partnership Joint Research Projects grant No. JPJSBP120209921.

\vskip 1cm


\bibliographystyle{model1a-num-names}
\bibliography{<your-bib-database>}

\begin{thebibliography}{00}
\bibitem{Aum}
 R. Aumann, Integrals of set-valued functions,
 {\small \it J. Math. Anal. Appl.} {\small \bf 12} (1965) 1-12.

 \bibitem{Ezzaki}
 F. Ezzaki and K. Tahri, Representation theorem of set valued regular martingale:
Application to the convergence of set valued martingale, {\small \it Statistics and Probability Letters} {\small \bf 154} (2019) 108-48.

 \bibitem{Hia}
 F. Hiai, and H. Umegaki. Integrals, conditional
expectations and martingales of multivalued functions. {\em J.
Multivar. Anal.,} {\bf 7} (1977): 149-182.

\bibitem{Kis2014}
M. Kisielewicz, Martingale representation theorem for set-valued martingales, {\small \it J. Math. Anal. Appl.} {\small \bf 409} (2014) 111-118.

\bibitem{Kis2015}
M. Kisielewicz, Properties of generalized set-valued stochastic integrals, {\small \it Discussiones Mathematicae
Differential Inclusions, Control and Optimization} {\small \bf  34} (2014) 131-147.

\bibitem{Jun}
 E.J. Jung and J.H. Kim, On set-valued stochastic
integrals, {\small\it Stoch Anal Appl} {\small\bf 21 (2)} (2003) 401- 418.



\bibitem{oksendal}
B. ${\O}$ksendal, Stochastic Differential Equations, Sixth Edition, Springer-Verlag, Berlin, Heildlberg, 2003.

\bibitem{sun}
Y. Sun, A. Han, Y. Hong and S. Wang, Threshold autugressive models for interval-valued time series data, {\small \it Journal of Economitrics} {\small \bf  (206) 2} (2018) 414-446.

\bibitem{thesis} J. Zhang, Integrals and Stochastic Differential
Equations for Set-Valued Stochastic Processes, PhD Thesis, Saga
University, Japan, 2009.

\bibitem{Zha09}
J. Zhang, S. Li, I. Mitoma and Y. Okazaki,  On set-valued stochastic
integrals in an M-type 2 Banach space, {\small\it J. Math. Anal. Appl.} {\small\bf 350}
 (2009) 216-233.

\bibitem{zhang20}
J. Zhang, I. Mitoma and Y. Okazaki, Submartingale property of set-valued stochastic integration associated with Poisson process and related integral equations on Banach spaces. arXiv:2002.09220 [math.PR]




\end{thebibliography}

\end{document}